\documentclass{amsproc}

\newtheorem{thm}{Theorem}[section]
\newtheorem{lemma}[thm]{Lemma}

\theoremstyle{definition}
\newtheorem{dfn}[thm]{Definition}

\theoremstyle{remark}

\numberwithin{equation}{section}

\newcommand{\Ric}{\ensuremath{\mathrm{Ric}}}

\begin{document}

\title{Complete Shrinking Ricci Solitons have Finite Fundamental Group}

\author{William Wylie}
\address{Department of Mathematics, University of California, Los Angeles, CA 90095 }

\email{wylie@math.ucla.edu}

\date{April 2, 2007}


\begin{abstract}
We show that if a complete Riemannian manifold supports a vector field such that the Ricci tensor plus the  Lie derivative of the metric with respect to the vector field  has a positive lower bound, then the fundamental group is finite.  In particular,  it follows that  complete shrinking Ricci solitons and complete smooth metric measure spaces with a positive lower bound on the  Bakry-Emery tensor  have finite fundamental group.  The method of proof is to generalize arguments of Garcia-Rio and  Fernandez-Lopez in the compact case.
\end{abstract}

\maketitle

\section{Introduction}
In this paper we are interested in studying complete Riemannian manifolds $(M,g)$ with a vector field $X$  such that, for some $\lambda > 0$,  
\begin{equation}
\label{MainThing} \Ric_g  + \mathcal{L}_Xg \geq \lambda g
\end{equation}
where $\mathcal{L}_Xg$ is the Lie derivative of $g$ with respect to the vector field $X$.  

This class  of manifolds includes two well known subclasses.  The first is the class of  complete manifolds such that $ \Ric_g  + \mathcal{L}_Xg = \lambda g$.  This is the class of complete Ricci Solitons.   The study of  Ricci solitons  is important in understanding many aspects of Ricci flow \cite{HRF}.  The second subclass  is the class of smooth metric measure spaces with Bakry-Emery tensor bounded below by $\lambda$. This class consists of smooth manifolds that satisfy (\ref{MainThing}) for some gradient vector field $X = \nabla f$.  We refer the reader to Lott \cite{Lott2003} and Qian \cite{Qian1997}  for  topological results concerning the Bakry-Emery tensor. 

 Fern\'{a}ndez-L\'{o}pez and Garc\'{i}a-R\'{i}o   \cite{FLGR} have  proven the following Myers type theorem: If a complete manifold $(M,g)$ satisfies (\ref{MainThing}) and the vector field $X$ has bounded norm,  then M is compact.  It is an immediate corollary that any compact manifold satisfying (\ref{MainThing}) has finite fundamental group.  In the case where $X$ is a gradient vector field this was proven earlier by Lott \cite{Lott2003}.  
 
The assumption of bounded $||X||$ is necessary to show $M$ is compact since, for example, Euclidean space with the vector field $X(v) = v$, $\forall v \in \mathbb{R}^n$ satisfies (\ref{MainThing}).  However, we show that  the fundamental group must be finite in the noncompact case.

\begin{thm} \label{Theorem} If $M$ is a complete Riemannian manifold satisfying (\ref{MainThing})  then $M$ has finite fundamental group. \end{thm}
 
Naber \cite{Naber2006} has  indepedently proven  Theorem  \ref{Theorem} under the additional assumption that $X = \nabla f$ and the Ricci curvature is bounded. 
 
There are complete manifolds with positive Ricci curvature and infinite fundamental group.  It is an obvious  consequence of Theorem \ref{Theorem} that the Ricci tensor of these manifolds can not be perturbed by a Lie derivative term to be bounded away from zero. 

The  proof of Theorem \ref{Theorem} is similar to the arguments of  Fern\'{a}ndez-L\'{o}pez and Garc\'{i}a-R\'{i}o \cite{FLGR}.  They show that if $M$ satisfies (\ref{MainThing}) and $|| X ||$ is bounded then the integral of the Ricci curvature along every geodesic is infinite.   By the Ambrose theorem \cite{Ambrose1957} this implies that the manifold  is compact.  The main idea of this paper is to replace the Ambrose theorem with an estimate of Hamilton  \cite{Hamilton1993} (See Lemma \ref{Variation} below).  This estimate  allows us to obtain an upper  bound on  the distance between two points  that  depends  only on the value of $||X||$ at each point and an upper bound on the Ricci curvature in a neighborhood of each point (Theorem \ref{Isom}).  Applying the upper bound  to a point in the universal cover of $M$ and its image under a deck transformation yields Theorem \ref{Theorem}.  

\section{Proof of Theorem \ref{Theorem}}
We make the following definition to aid the exposition. 
\begin{dfn} For any point $p \in M$ define \[H_p =  \max \left\{ 0, \sup \left \{ \Ric_y(v,v) : y \in B(p,1), ||v||=1\right\} \right \}.\] \end{dfn}
We can now state the main lemma.
\begin{lemma} \label{Variation}  Let $(M,g)$ be a complete Riemannian manifold, let $p,q \in M$ such that $r= d(p,q) > 1$ and let $\gamma$ be the minimal geodesic from $p$ to $q$ parametrized by arclength, then 
\[ \int_0^{r} \Ric (\gamma'(s), \gamma'(s))ds  \leq 2(n-1)+H_p + H_q. \]
 \end{lemma}
Lemma \ref{Variation}  was used by  Hamilton \cite{Hamilton1993} to study  the change in the distance function on a Riemannian manifold evolving by Ricci flow and also appears in Perelman (\cite{Per1}, Lemma 8.1).  We include the proof for completeness.
\begin{proof}
 By the second variation of arclength formula, for any piecewise smooth function $\phi$ with $\phi(0) = \phi(r) = 0$,
\begin{equation} \label{phi} 0 \leq \int_0^r \left((n-1)(\phi '(s))^2 - \phi^2(s)\Ric(\gamma'(s), \gamma'(s)) \right)ds. \end{equation}
Let $\phi$ be the function
\[\phi(s) =  \left\{
\begin{array}{ll}
s & \mbox{$0\leq s \leq 1$} \\
1 & \mbox{$1 \leq s \leq r-1$}\\
r-s& \mbox{ $r-1 \leq s \leq r$}
\end{array}
\right. \]
Then, since $\phi(s) = 1$ and $\phi '(s) = 0$ for $1 \leq s \leq r-1$, (\ref{phi}) becomes  
\begin{eqnarray*}
0 &\leq& \int_{0}^1 \left((n-1) (\phi'(s))^2\right) ds + \int_{r-1}^r \left((n-1) (\phi'(s))^2\right) ds -  \int_{0}^1\left(\phi^2(s)\Ric(\gamma', \gamma')\right)ds \\ &&- \int_{r-1}^r \left(\phi^2(s)\Ric(\gamma'(s), \gamma'(s))\right) ds - \int_{1}^{r-1}\Ric(\gamma'(s), \gamma'(s))ds.
\end{eqnarray*}
Adding $\int_0^r \Ric(\gamma'(s), \gamma'(s)) ds$ to both sides of the equation yields 
\begin{eqnarray}
\int_0^r  \Ric(\gamma'(s), \gamma'(s)) ds &\leq& \int_{0}^1\left( (n-1) (\phi'(s))^2 \right)ds + \int_{r-1}^r \left((n-1)(\phi'(s))^2\right) ds \nonumber \\ \label{Int}&&+  \int_{0}^1\left(\left(1-\phi^2(s)\right)\Ric(\gamma'(s), \gamma'(s))\right)ds \\ &&+ \int_{r-1}^r \left(\left(1-\phi^2(s)\right)\Ric(\gamma'(s), \gamma'(s)) \right)ds. 
\nonumber\end{eqnarray}
We now work out the terms on the right hand side of equation (\ref{Int}).  Since $|\phi'(s)| =1$ for $0 \leq s \leq 1$ and $r-1 \leq s \leq r$ we have
\begin{equation} \label{Term1} \int_{0}^1\left( (n-1) (\phi'(s))^2 \right)ds + \int_{r-1}^r \left((n-1)(\phi'(s))^2\right) = 2(n-1). \end{equation}
Moreover,  $0 \leq \phi \leq 1$ and $\Ric(\gamma'(s), \gamma'(s)) \leq H_p$ for $0\leq s\leq 1$, therefore
\begin{equation} \label{Term2}\int_{0}^1\left(\left(1-\phi^2(s)\right)\Ric(\gamma'(s), \gamma'(s))\right)ds. \leq H_p. \end{equation}
Similarly, since  $\Ric(\gamma'(s), \gamma'(s)) \leq H_q$ for $r-1 \leq s \leq r$, 
\begin{equation} \label{Term3}\int_{r-1}^r \left(\left(1-\phi^2(s)\right)\Ric(\gamma'(s), \gamma'(s)) \right)ds. \leq H_q. \end{equation}
Thus, combining (\ref{Int}), (\ref{Term1}), (\ref{Term2}), and (\ref{Term3}) gives the lemma. 

\end{proof}

Using  Lemma \ref{Variation} and the arguments in \cite{FLGR},  we can now derive an upper bound on the distance between two points that depends only on $||X||$ and $H$.  

\begin{thm} \label{Isom}If $(M,g)$ is  a complete manifold satisfying (\ref{MainThing}) then, for any $p,q \in M$,
\begin{equation}  \label{IsomEq} d(p, q) \leq \max \left \{ 1,  \frac{1}{\lambda} \big ( 2(n-1) + H_p+ H_q + 2||X_p|| + 2||X_q|| \big ) \right\}.  \end{equation}
\end{thm}

\begin{proof}
Assume that $d(p, q) >1$ and let $\gamma$ be the minimal geodesic from $p$ to $q$.  Applying  Lemma \ref{Variation} we have 
\begin{equation} \label{Old}  \int_0^{r} \Ric (\gamma'(s), \gamma'(s)) ds \leq 2(n-1)+H_p + H_q. \end{equation}
On the other hand, by equation (\ref{MainThing})  
\begin{eqnarray}  \int_0^{r} \Ric (\gamma'(s), \gamma'(s)) ds  &\geq&  \int_0^{r}\left( \lambda g(\gamma'(s), \gamma'(s)) - \mathcal{L}_Xg(\gamma'(s), \gamma'(s)) \right) ds \nonumber \\ &\geq & \label{Deriv} \lambda d(p, q) + 2 g_p(X, \gamma'(0)) - 2 g_{q}(X, \gamma'(r)) \nonumber \\ & \geq &  \label{Done} \lambda d(p, q) -  2||X_p|| - 2||X_q||.  
\end{eqnarray}
Where, in the last step, we have used  \[ \displaystyle \mathcal{L}_Xg\left(\gamma'(s), \gamma'(s)\right) = 2 \frac{d}{ds} g(X, \gamma'(s)).\]  Combining  (\ref{Old}) and (\ref{Done}) and solving for $d(p,q)$ gives (\ref{IsomEq}).
\end{proof}

\begin{proof} [Proof of Theorem \ref{Theorem}]
Let $\widetilde{M}$ be the universal cover of $M$.  $\widetilde{M}$ satisfies (\ref{MainThing})  for the pullback metric and  pullback vector field, $\widetilde{X}$. Fix $\tilde{p}$  in $\widetilde{M}$ and let $h \in \pi_1(M)$ identified as a deck transformation on $\widetilde{M}$. Note that $B(\tilde{p},1)$ and $B(h(\tilde{p}), 1)$ are isometric, thus $H_p = H_{h(p)}$.   Also, $||\widetilde{X}_{\widetilde{p}}|| = || \widetilde{X}_{h(\widetilde{p})}|| $ so by applying Theorem \ref{Isom} to the points $\tilde{p}$ and $h(\tilde{p})$ we obtain 
\[ d(\widetilde{p}, h(\widetilde{p})) \leq \max \left \{ 1,  \frac{2}{\lambda} \left(n-1 + H_{\tilde{p}} + 2||\widetilde{X}_{\tilde{p}}|| \right) \right\}   \qquad \forall h \in \pi_1(M).\]
  Since the right hand side is independent of $h$,  this proves the theorem.

\end{proof}

\textbf{Acknowledgements:} I would like to thank the authors of \cite{FLGR} for providing me with a copy of their work,  Ben Chow for encouraging me to study Ricci solitons and for his interest in this work, and  Guofang Wei and Peter Petersen for many helpful discussions.  This work was partially completed while at MSRI.

\bibliographystyle{amsalpha}

\end{document}